\documentclass[11pt]{article}
\usepackage{amsmath}
\usepackage{amssymb}
\usepackage{latexsym}
\usepackage{amsthm}
\usepackage{mathrsfs}
\usepackage{fancyhdr}
\usepackage{amsfonts}
\usepackage{amssymb}
\usepackage{graphicx}
\usepackage{latexsym, bm}

\newtheorem{thm}{Theorem}[section]
\newtheorem{lem}[thm]{Lemma}
\newtheorem{cor}[thm]{Corollary}

\newtheorem{conj}[thm]{Conjecture}

\DeclareGraphicsExtensions{.eps,.eps.gz}
\title{A Maximum Resonant Set of Polyomino Graphs
\thanks{This work is supported by NSFC (grant no. 11371180).}}
\author{Heping Zhang\footnote{Corresponding author.}, \quad Xiangqian Zhou
\\\small{ School of Mathematics and Statistics, Lanzhou University, Lanzhou, Gansu 730000, P. R. China}
\\\small{E-mail addresses: zhanghp@lzu.edu.cn, ~zhouxiangqian0502@126.com}}
\date{}

\begin{document}

\maketitle


\begin{abstract}

A polyomino graph $H$ is a connected finite subgraph of the infinite plane grid such that
each finite face is surrounded by a regular square of side length one and each edge belongs to at least one square.
In this paper, we show that if $K$ is a maximum resonant set of $H$, then $H-K$ has a unique perfect matching.
We further prove that the maximum forcing number of a polyomino graph is equal to its Clar number.
Based on this result, we have that the maximum forcing number of a polyomino graph
can be computed in polynomial time. We also show that if $K$ is a maximal alternating set of $H$, then $H-K$ has a unique perfect matching.

\vspace{0.3cm}

\end{abstract}


\section{Introduction}

There are two families of interesting plane bipartite graphs, hexagonal systems \cite{19} and
polyomino graphs \cite{26}, which often arise in some real-world problems.
A hexagonal system with a perfect matching is viewed as the carbon-skeleton of a benzenoid hydrocarbon \cite{4,7}.
The dimer problem in crystal physics is to count perfect matchings of polyomino graphs \cite{10}.

Consider an $m\times n$ chessboard $P$ (i.e., one special type of polyomino graphs) with $mn$ even.
In how many ways can this board be covered by dominoes ($1\times 2$ rectangles or dimers)?
In graph theoretics terms, this problem is equivalent to the counting of perfect matching of its inner
dual graph (for the definition of inner dual graph, see \cite{26}).
A \emph{domino pattern} of  $P$ is obtained by paving or tiling dominoes.
There is a one-to-one correspondence between the domino patterns of $P$ and perfect matchings of
the inner dual of $P$. Kasteleyn \cite{10} developed a so-called ``Pfaffian method'' and
derived the explicit expression of the number of ways for covering $P$ by dominoes.

In general, perfect matching existence \cite{27}, elementary components \cite{Ke, 22}, matching forcing number \cite{16} and maximal resonance \cite{13} of polyomino graphs have been investigated.
In addition, polyomino graphs are also models of many interesting combinatorial subjects, such as hypergraphs
\cite{2}, domination problem \cite{5,6}, rook polynomials \cite{15}, etc.

The concept of resonance originates from the conjugated circuits method \cite{9,21}.
Conjugated or resonant circuits (i.e. alternating cycles in mathematics \cite{14})
appeared in Clar's aromatic sextet theory \cite{4} and Randi\'{c}'s conjugated circuit model
\cite{17,18}. Klein \cite{11} emphasized the connection of Clar's ideas with the conjugated circuits method.
In the Clar's aromatic sextet theory \cite{4}, Clar found that an aromatic hydrocarbon molecule with larger number
of mutually resonant hexagons is more stable.

Let $G$ be a plane bipartite graph with a  \emph{perfect matching} (or Kekul\'{e} structure in chemistry)
$M$ (i.e. a set of independent edges covering all vertices of $G$).
A cycle of $G$ is called an $M$-\emph{alternating cycle} if its edges appear alternately in $M$ and off $M$.
A face  $f$ is said to be $M$-\emph{resonant} or -{\em alternating} if its boundary is an $M$-alternating cycle.
Let $\mathcal{H}$ be a  set  of finite faces (the intersection is allowed) of $G$. $\mathcal{H}$ is called an $M$-\emph{alternating set} if all faces in $\mathcal{H}$ are $M$-resonant.
Further, $\mathcal{H}$ is called an $M$-\emph{resonant set}  of $G$ if the faces in $\mathcal{H}$ are mutually disjoint and $\mathcal{H}$ is an $M$-alternating set.
Simply, $\mathcal{H}$ is a \emph{resonant set} and alternating set of $G$ if $G$ has a perfect matching $M$ such that $\mathcal{H}$ is an $M$-resonant set and $M$-alternating set respectively.
The cardinality of a maximum resonant set of $G$ is called the \emph{Clar number} of $G$, denoted by $Cl(G)$.

In 1985, Zheng and Chen  \cite{28}  gave an important property  for  a maximum resonant set of a hexagonal system.

\begin{thm}\label{thm-1-1}
\emph{\cite{28}}
Let $H$ be a hexagonal system and $K$ a maximum resonant set of $H$.
Then $H-K$ has a unique perfect matching.
\end{thm}

A \emph{forcing set} of a perfect matching $M$ of a graph $G$ is a subset $S\subseteq M$ such that $S$
is not contained in any other perfect matching of $G$.
The \emph{forcing number} of a perfect matching $M$, denoted by $f(G,M)$,
is the cardinality of a minimum forcing set of $M$. The \emph{maximum forcing number} of $G$ is
the maximum value of forcing numbers of all perfect matchings of $G$, denoted by $F(G)$.
The concept of forcing number of graphs was originally introduced for benzenoid systems by Harary et al. \cite{8}.
The same idea appeared in an earlier paper \cite{12} of Klein and Randi\'{c}
by the name ``innate degree of freedom''. Most known results on forcing number are referred to \cite{3}.

Pachter and  Kim revealed a minimax result that connects the forcing number of a perfect matching and its alternating cycles
as follows.

\begin{thm}\label{thm-1-2}
\emph{\cite{16}} Let $G$ be a plane bipartite graph with a perfect matching.
Then for any perfect matching $M$ of $G$, $f(G,M)=c(M)$,
where $c(M)$ denote the maximum number of disjoint $M$-alternating cycles in $G$.
\end{thm}

By combining Theorems \ref{thm-1-1} and \ref{thm-1-2}, Xu et al. \cite{23}
obtained a relation between the forcing number and Clar number of a hexagonal system as follows.

\begin{thm}\label{thm-1-3}
\emph{\cite{23}} Let $H$ be an elementary hexagonal system. Then $F(H)=Cl(H)$.
\end{thm}

An alternating set of a graph $G$ is called \emph{maximal}
if it is not properly contained in another alternating set of $G$.
In 2006, Salem and Abeledo obtained the following result.

\begin{thm}
\emph{\cite{20}}
Let $H$ be a hexagonal system and $K$ a maximal alternating set of $H$.
Then $H-K$ has a unique perfect matching.
\end{thm}

Motivated by the above work, we will naturally  investigate polyomino graphs.
This paper is mainly concerned with a maximum resonant set of a polyomino graph. By applying Zheng and Chen's approach \cite{28},  we prove that if $K$ is a maximum resonant set of a polyomino graph $G$,
then $G-K$ has a unique perfect matching. For a maximal alternating set
of $G$, this property still holds.
As a corollary,  we have that the maximum forcing number of a polyomino graph is equal to its Clar number.
Based on these results, it can be shown that
the maximum forcing number of a polyomino graph can be computed in a polynomial time, and thus conforms
the conjecture proposed by Xu et al. \cite{23}.


\section{Maximum resonant set}

A \emph{polyomino graph} is a connected finite subgraph of the infinite plane grid
such that each interior face is surrounded by a regular square of side length one
and each edge belongs to at least one square \cite{26}. For a polyomino graph $H$,
the boundary of the infinite face of $H$ is called the \emph{boundary} of $H$,
denoted by $\partial(H)$, and each edge on the boundary is called a \emph{boundary edge} of $H$.
It is well known that polyomino graphs are bipartite graphs. For convenience,
we always place a polyomino graph considered on a plane so that one of the
two edge directions is horizontal and the other is vertical.
Two squares are \emph{adjacent} if they have an edge in common.
A vertex of $H$ lying on the boundary of $H$ is called an \emph{external vertex},
and a vertex not being external is called an \emph{internal vertex}.
A square of $H$ with external vertices is called an \emph{external square}, and a square with no external
vertices is called an \emph{internal square}.
In what follows, we always restrict our attention to polyomino graphs with perfect matchings.

Let $G$ be a graph with a perfect matching $M$ and an $M$-alternating cycle $C$. Then
$M\oplus C(=M\oplus E(C))$ is also a perfect matching of $G$ and $C$ is an
$(M\oplus C)$-alternating cycle of $G$ \cite{27}. Let $M$ and $N$ be two perfect matchings of a graph $G$.
The \emph{symmetric difference} of $M$ and $N$, denoted by $M\oplus N$, is the set of edges contained in
either $M$ or $N$, but not in both, i.e., $M\oplus N=(M\cup N)-(M\cap N)$.
An $(M,N)$-alternating cycle of $G$ is a cycle whose edges are
in $M$ and $N$ alternately. It is well known that the symmetric difference of two perfect matchings
$M$ and $N$ of $G$ is a disjoint union of $(M,N)$-alternating cycles.

We now state our main result as follows.

\begin{thm}\label{thm-2-1}
Let $H$ be a polyomino graph with a perfect matching, and $K$ be a maximum resonant set of $H$.
Then $H-K$ has a unique perfect matching.
\end{thm}

Before proving the main theorem, we will deduce the following crucial lemma.

Let $G$ be a plane bipartite graph, $K$ a set of finite faces and $H$ a subgraph of $G$.
By $K\cap H$ we always mean the intersection of  $K$ and the set of faces of $H$.

\begin{lem}\label{lem-2-2}
Let $H$ be a $2$-connected  polyomino graph with a perfect matching,
$K$ a resonant set consisting of internal squares of $H$.
If $H-K-\partial(H)$ has a perfect matching or is an empty graph, then $K$ is not a  maximum resonant set.
\end{lem}

\begin{figure}[!htbp]
\begin{center}
\includegraphics[height=8.35cm]{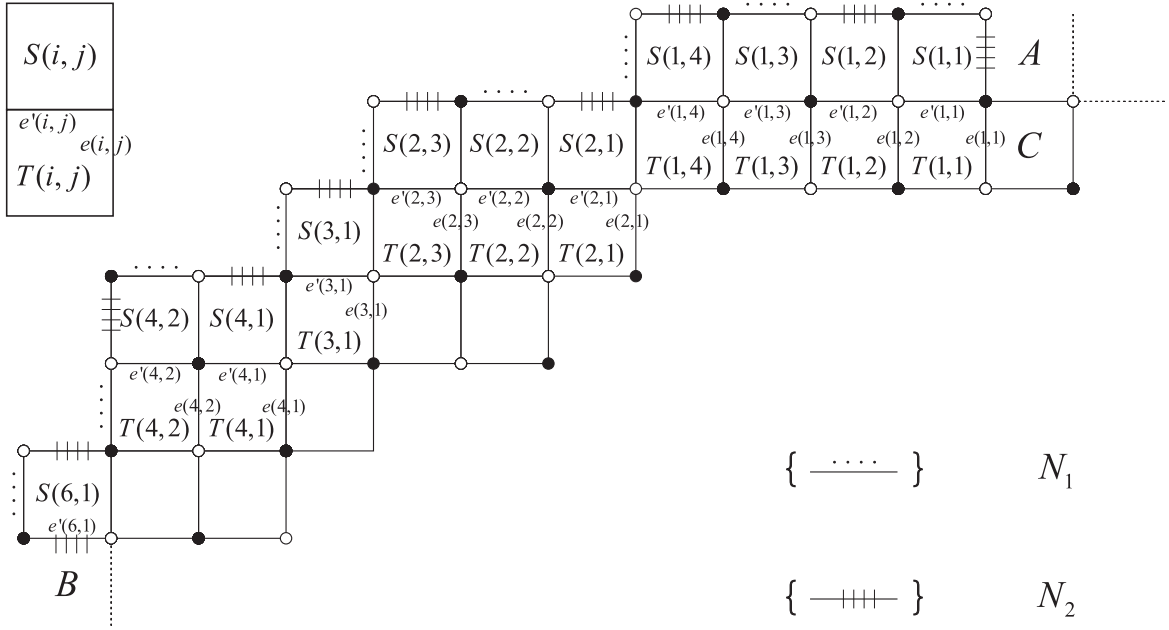}
\caption{{\small Squares $S(i,j)$ and $T(i,j)$,
and edges $e(i,j)$ and $e'(i,j)$ with $m=6$,
n(1)=4, n(2)=3, n(3)=1, n(4)=2, n(5)=0, n(6)=1.
And $A,B\notin H$.}}\label{fig:polyomino-01}
\end{center}
\end{figure}

\begin{proof} We denote by $M$ a perfect matching of $H-K-\partial(H)$, and decompose the edges of
$\partial(H)$ into two perfect matchings $N_1$ and $N_2$ of $\partial(H)$ since $\partial(H)$ is an even cycle.
Then it is clear that $M\cup N_1$ and $M\cup N_2$ are two perfect matchings of $H-K$. Let $M'$ be a perfect matching
of $K$ such that each edge of $M'$ is vertical. Then $M\cup M'$ is a perfect matching of $H-\partial(H)$.
Moreover, $M_1:=N_1\cup (M\cup M')$ and $M_2:=N_2\cup (M\cup M')$ are two perfect matchings of $H$.

Suppose to the contrary that $K$ is a maximum resonant set of $H$.
Adopting the notations of \cite{28}, we can take a series of external squares $\{S(i,j): 1\leq i\leq m, 1\leq j\leq n(i)\}$
which  satisfy that neither square $A$ nor square $B$
is contained in $H$ as shown in  Fig. \ref{fig:polyomino-01}. We denote edges, if any, by $e(i,j)$, $e'(i,j)$,
$1\leq i\leq m$ and $1\leq j\leq n(i)$, and denote the square with edge  $e'(i,j)$ which is adjacent to
$S(i,j)$, if any, by $T(i,j)$, $1\leq i\leq m$ and $1\leq j\leq n(i)$, as shown in Fig. \ref{fig:polyomino-01}.
We first prove the following claims.
\vskip 0.2cm
\noindent\textbf{Claim 1.}  For a pair of parallel edges $e_1$ and $e_2$ of a square $s$ of $H$,
they do not lie simultaneously on the boundary of $H$.
\vskip 0.2cm

If $e_1$ and $e_2$ lie on the boundary of $H$, then $\{e_1, e_2\}\subseteq N_1$ or $N_2$, say $N_1$.
So the square $s$ is $N_1$-alternating, and $K\cup \{s\}$ is a resonant set of $H$,
which contradicts that $K$ is a maximum resonant set of $H$. Hence Claim $1$ holds.
\vskip 0.2cm
\noindent\textbf{Claim 2.}  $n(1)\geq 2$ is even, $n(m)=0$, the square $C\in H$, $n(2)\geq 1$
and $m\geq 3$. If $n(i)>0$, then $T(i,j)\in H$ for all $j$, $1\leq j\leq n(i)$.
\vskip 0.2cm

Claim 1 implies that $n(1)\geq 2$, $n(m)=0$,
and $T(i,j)\in H$ for all $1\leq j\leq n(i)$.
It remains to show that $n(1)$ is even, $C\in H$, $n(2)\geq 1$ and $m\geq 3$.

Since $K$ is a maximum resonant set of $H$,
$e'(i,j)\notin M$ for all $1\leq j\leq n(i)$, $1\leq i\leq m$.
Otherwise, $K\cup \{S(i,j)\}$  is a resonant set of $H$
since the square $S(i,j)$ is either $M_1$-alternating or $M_2$-alternating, a contradiction.
So $e(1,j)\in M\cup M'$ for all $2\leq j\leq n(1)$.

First, we show that $n(1)$ is even.

Suppose to the contrary that $n(1)$ is odd with $n(1)\geq 3$ (See Fig. \ref{fig:polyomino-02}).
We use $H_0$ to denote the subgraph of $H$ formed by squares
$S(1,1),S(1,2),\ldots,S(1,n(1)),T(1,2),T(1,3),\ldots,T(1,n(1)-1)$.
Then we can see that the restriction of $M_2$ on  $H_0$ is a perfect matching of $H_0$.
Let $M_2'=M_2\oplus T(1,2)\oplus T(1,4)\oplus \cdots\oplus T(1,n(1)-1)
\oplus S(1,2)\oplus S(1,4)\oplus \cdots\oplus S(1,n(1)-1)$.
Then $M_2'$ is a perfect matching of $H$ such that each member in the set
$$S_0:=\Big(K\cup \{S(1,1),S(1,3),\ldots,S(1,n(1))\}\Big)\Big\backslash \Big(K\cap H_0\Big)$$
is an $M_2'$-alternating square.
Note that the set $\{S(1,1),S(1,3),\ldots,S(1,n(1))\}$ with cardinality
$\frac{n(1)+1}{2}$, whereas $|K\cap H_0|\leq \frac{n(1)-1}{2}$.
Hence,   $S_0$ is a resonant set of $H$ larger than $K$.
This contradicts that $K$ is a maximum resonant set of $H$.

\begin{figure}[!htbp]
\begin{center}
\includegraphics[height=3.74cm]{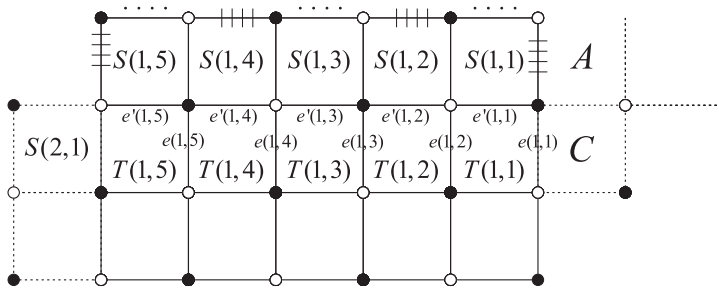}
\caption{{\small Illustration for Claim 2 in the proof of Lemma \ref{lem-2-2}:
$n(1)$ is odd.}}\label{fig:polyomino-02}
\end{center}
\end{figure}

Next, we show that $C\in H$.

Suppose to the contrary that $C\notin H$. Then $e(1,1)\in N_1$.
We use $H_1$ to denote the subgraph of $H$ formed by squares
$T(1,1),T(1,2),\ldots,T(1,n(1)-1)$.
Then the restriction of $M_1$ on  $H_1$ is a perfect matching of $H_1$.
Note that  $T(1,1),T(1,3),\ldots,T(1,n(1)-1)$ are $M_1$-alternating squares and $|K\cap H_1|\leq \frac{n(1)}{2}-1$.
Hence,  we can see that
$$\Big(K\cup \{T(1,1),T(1,3),\ldots,T(1,n(1)-1)\}\Big)\Big\backslash \Big(K\cap H_1\Big)$$
is a resonant set of $H$ larger than $K$, a contradiction.
Similarly, $S(2,1)\in H$ and $n(2)>0$. Moreover, $m\geq 3$.
So we complete the proof of Claim $2$.

Let $\ell$ be an integer with $2\leq \ell\leq m$ such that $n(\ell)$ is even,
and $n(t)$ is odd for all $2\leq t\leq \ell-1$.
It follows from $e(1,n(1))\in M\cup M'$, $e'(i,j)\notin M$ that
$e(i,j)\in M\cup M'$ for all $i$ and $j$, $2\leq i \leq \ell-1$, $1\leq j\leq n(i)$.
We now need to distinguish the following two cases.

\begin{figure}[!htbp]
\begin{center}
\includegraphics[height=4.91cm]{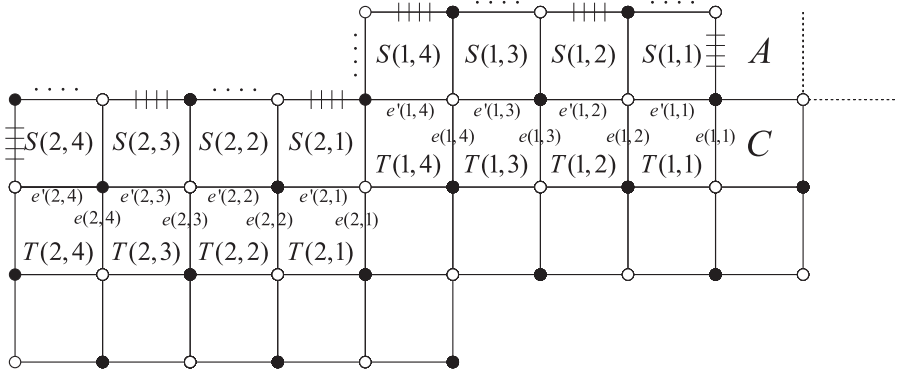}
\caption{{\small Illustration for Subcase 1.1 in the proof of Lemma \ref{lem-2-2}:
$\ell=2$ and $n(\ell)=4$.}}\label{fig:polyomino-03}
\end{center}
\end{figure}

\begin{figure}[!htbp]
\begin{center}
\includegraphics[height=6.08cm]{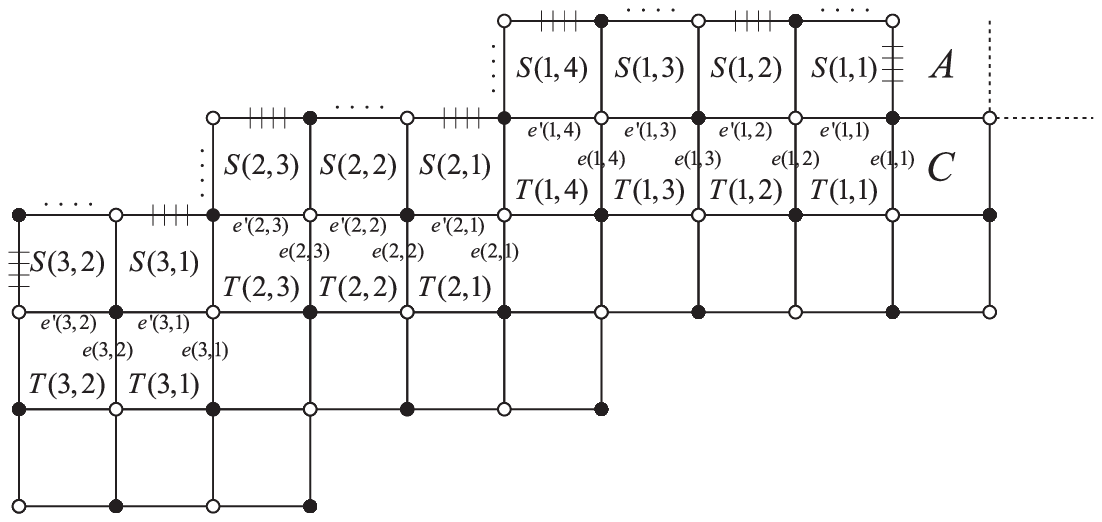}
\caption{{\small Illustration for Subcase 1.2 in the proof of Lemma \ref{lem-2-2}:
$\ell=3$, $n(\ell)=2$.}}\label{fig:polyomino-04}
\end{center}
\end{figure}

{\bf Case 1.} $n(\ell)>0$.
In this case we have $e(\ell,j)\in M\cup M'$ for all $1\leq j\leq n(\ell)$.

\emph{Subcase 1.1:} $\ell=2$ (See Fig. \ref{fig:polyomino-03}).
Let $H_2$ denote the subgraph of $H$ formed by squares in the set
$$\{T(1,j)\mid 2\leq j\leq n(1)\}\cup \{S(2,j)\mid 1\leq j\leq n(2)\}\cup
\{T(2,j)\mid 1\leq j\leq n(2)-1\}.$$
Then the restriction of $M_2$ on  $H_2$ is a perfect matching of $H_2$.
Let
$$M_2''=M_2\oplus T(2,1)\oplus T(2,3)\oplus \cdots\oplus T(2,n(2)-1)
\oplus S(2,1)\oplus S(2,3)\oplus \cdots\oplus S(2,n(2)-1),$$
$$S_1:=\{T(1,2),T(1,4),\ldots,T(1,n(1))\}\cup \{S(2,2),S(2,4),\ldots,S(2,n(2))\}.$$
Then $M_2''$ is a perfect matching of $H$ such that each member of
$(K\cup S_1)\backslash (K\cap H_2)$
is an $M_2''$-alternating square. Note that
$|S_1|=\frac{n(1)+n(2)}{2}$, whereas
$$|K\cap \{T(1,j)\mid 2\leq j\leq n(1)\}|\leq \frac{n(1)}{2}-1,$$
$$|K\cap \{T(2,j)\mid 1\leq j\leq n(2)-1\}|\leq \frac{n(2)}{2}.$$
Hence,  $(K\cup S_1)\backslash (K\cap H_2)$
is a resonant set of $H$ larger than $K$, a contradiction.

\emph{Subcase 1.2:} $\ell\geq 3$ (See Fig. \ref{fig:polyomino-04}).
Let $H_3$ denote the subgraph of $H$ formed by squares in
$$\{T(\ell-1,j)\mid 1\leq j\leq n(\ell-1)\}\cup
\{S(\ell,j)\mid 1\leq j\leq n(\ell)\}\cup
\{T(\ell,j)\mid 1\leq j\leq n(\ell)-1\}.$$
Then the restriction of $M_2$ on  $H_3$ is a perfect matching of $H_3$.
Let
$$M_2'''=M_2\oplus T(\ell,1)\oplus T(\ell,3)\oplus \cdots\oplus T(\ell,n(\ell)-1)
\oplus S(\ell,1)\oplus S(\ell,3)\oplus \cdots\oplus S(\ell,n(\ell)-1),$$
$$S_2:=\{S(\ell,2),S(\ell,4),\ldots,S(\ell,n(\ell))\}\cup \{T(\ell-1,1),T(\ell-1,3),\ldots,T(\ell-1,n(\ell-1))\}.$$
Then $M_2'''$ is a perfect matching of $H$ such that each member of
$(K\cup S_2)\backslash (K\cap H_3)$ is an $M_2'''$-alternating square.
Note that
$|S_2|=\frac{n(\ell)+n(\ell-1)+1}{2}$, whereas
$$|K\cap \{T(\ell-1,j)\mid 1\leq j\leq n(\ell-1)\}|\leq \frac{n(\ell-1)-1}{2},$$
$$|K\cap \{T(\ell,j)\mid 1\leq j\leq n(\ell)-1\}|\leq \frac{n(\ell)}{2}.$$
Hence, $(K\cup S_2)\backslash (K\cap H_3)$ is a resonant set of $H$ larger than $K$, a contradiction.

\begin{figure}[!htbp]
\begin{center}
\includegraphics[height=6.08cm]{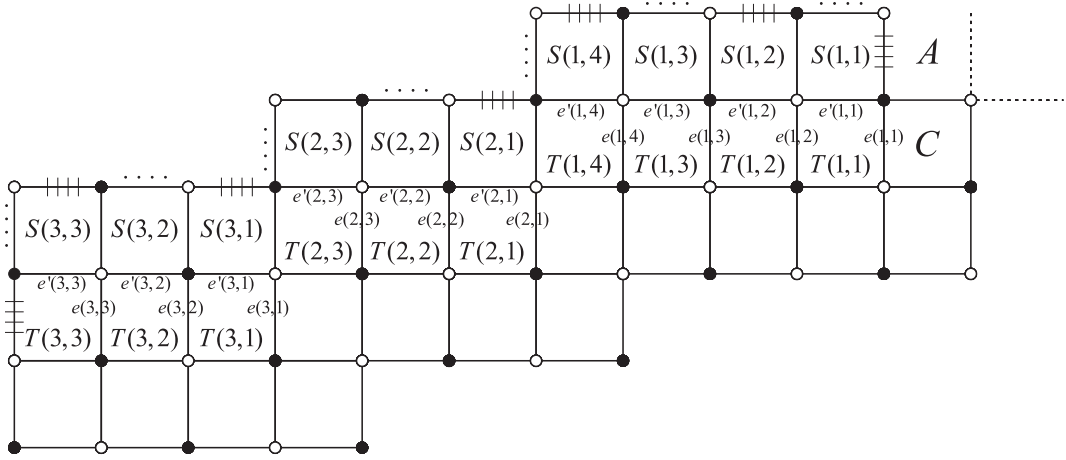}
\caption{{\small Illustration for Case 2 in the proof of Lemma \ref{lem-2-2}:
$\ell=4$, $n(\ell)=0$.}}\label{fig:polyomino-05}
\end{center}
\end{figure}

{\bf Case 2.} $n(\ell)=0$ (See Fig. \ref{fig:polyomino-05}).

Let $H_4$ denote the subgraph of $H$ formed by squares
$T(\ell-1,1),T(\ell-1,2),\ldots,T(\ell-1,n(\ell-1))$.
Note that the left side vertical edge of the square
$T(\ell-1, n(\ell-1))$ belongs to $N_2$, and moreover
each square in $H_4$ is $M_2$-alternating.
Thus we can see that
$$\Big(K\cup \{T(\ell-1,1),T(\ell-1,3),\ldots,T(\ell-1,n(\ell-1))\}\Big)\Big\backslash
\Big(K\cap H_4\Big)$$ is a resonant set of $H$ larger than $K$, a contradiction.

Now the entire proof of the lemma is complete.
\end{proof}

\vskip 0.2cm
\noindent\textbf{Proof of Theorem \ref{thm-2-1}.}
Suppose to the contrary that $H-K$ has two perfect matchings $M$ and $M'$. Then $M\oplus M'$ contains an
$(M,M')$-alternating cycle $C$. Let $I[C]$ denote the subgraph of $H$ consisting of $C$
together with its interior. Put $K^\ast=K\cap I[C]$. Then $K^\ast$ is not a maximum resonant set of $I[C]$
since $I[C]$ and $K^\ast$ satisfy the condition of Lemma \ref{lem-2-2}.
Moreover $K$ is also not a maximum resonant set of $H$,
which contradicts the assumption that $K$ is a maximum resonant set of $H$.
\hfill $\square$

\section{Maximal alternating set}
For a maximal alternating set of polyomino graphs, we can obtain the following results.

\begin{lem}\label{lem-3-1}
Let $H$ be a $2$-connected  polyomino graph with a perfect matching,
$K$ an alternating set consisting of internal squares and $\partial(H)$ the boundary of $H$.
If $H-K-\partial(H)$ has a perfect matching or is an empty graph,
then $K$ is not a maximal alternating set.
\end{lem}

\begin{proof}
We  adopt the substructure and notations used in the proof of Lemma \ref{lem-2-2}
(See Fig. \ref{fig:polyomino-01}).
Let $M$ be a perfect matching of $H-\partial(H)$ such that all squares in $K$ are $M$-alternating,
and  let $N_1$ and $N_2$ be two perfect matchings of $\partial(H)$.
Then  $M_1:=M\cup N_1$ and $M_2:=M\cup N_2$ are two perfect matchings of $H$.

Suppose to the contrary that $K$ is a  maximal alternating set of $H$.
The following Claim 1 and its proof are the same as Claim 1 of Lemma \ref{lem-2-2}.

\vskip 0.2cm
\noindent\textbf{Claim 1.}  For a pair of parallel edges $e_1$ and $e_2$ of a square $s$ of $H$,
they do not lie simultaneously on the boundary of $H$.
\vskip 0.2cm

\noindent\textbf{Claim 2.}  $n(1)\geq 2$, $n(m)=0$, the square $C\in H$, $n(2)\geq 1$
and $m\geq 3$. $e'(i,j)\notin M$ for all $1\leq j\leq n(i)$, $1\leq i\leq m$. Moreover,
$e(1,j)\in M$ for all $2\leq j\leq n(1)$.
\vskip 0.2cm

Claim 1 implies that $n(1)\geq 2$, $n(m)=0$,
and $T(i,j)\in H$ for all $1\leq j\leq n(i)$, $1\leq i\leq m$.

Since $K$ is a maximal alternating set of $H$,
$e'(i,j)\notin M$ for all $1\leq j\leq n(i)$, $1\leq i\leq m$.
Otherwise, $K\cup \{S(i,j)\}$  is an alternating set of $H$
since the square $S(i,j)$ is either $M_1$-alternating or $M_2$-alternating, a contradiction.
So $e(1,j)\in M$ for all $2\leq j\leq n(1)$.

Now we show that $C\in H$.
Suppose to the contrary that $C\notin H$. Then $e(1,1)\in N_1$ and $S(1,1)$ is $M_1$-alternating.
So $K\cup \{S(1,1)\}$  is an alternating set of $H$, a contradiction.
Symmetrically, $S(2,1)\in H$ and $n(2)\geq 1$. So $m\geq 3$.
Hence Claim $2$ is proved.

\vskip 0.2cm
Let $\ell$ be an integer with $3\leq \ell\leq m$ such that $n(\ell)=0$,
and $n(t)>0$ for all $2\leq t\leq \ell-1$.
It follows from $e(1,n(1))\in M$, $e'(i,j)\notin M$ that
$e(i,j)\in M$ for all $i$ and $j$, $2\leq i \leq \ell-1$, $1\leq j\leq n(i)$.
Note that the left  vertical edge of the square
$T(\ell-1, n(\ell-1))$ belongs to $N_1$ or $N_2$, say $N_1$.
So $T(\ell-1, n(\ell-1))$ is $M_1$-alternating and $K\cup \{T(\ell-1, n(\ell-1))\}$ is an alternating set of $H$,
which contradicts the assumption that $K$ is a maximal alternating set of $H$.
The lemma is proved.
\end{proof}

\begin{thm}\label{thm-3-2}
Let $H$ be a polyomino graph with a perfect matching, and $K$ be a maximal alternating  set of $H$.
Then $H-K$ has a unique perfect matching.
\end{thm}
\begin{proof}
Suppose to the contrary that $H-K$ has two perfect matchings $M$ and $M'$. Then $M\oplus M'$ contains an
$(M,M')$-alternating cycle $C$.  Put $K^\ast=K\cap I[C]$. Then $K^\ast$ is not a maximal alternating set of $I[C]$
since $I[C]$ and $K^\ast$ satisfy the condition of Lemma \ref{lem-3-1}.
So $K$ is also not a maximal alternating set of $H$, a contradiction.
\end{proof}

\section{Maximum forcing number}

Motivated by Theorem \ref{thm-1-3}, it is natural to ask the following question:
when does the maximum forcing number of a plane bipartite graph equal its Clar number?
In the following, we shall give a sufficient condition.

Let $G$ be a plane graph with a perfect matching. A cycle $C$ of $G$ is said to be \emph{nice}
if $G$ has a perfect matching $M$ such that $C$ is an $M$-alternating cycle.
Denote by $I[C]$ the subgraph of $G$ consisting of $C$ together with its interior.
A cycle $C$ of $G$ is called a \emph{face cycle} if it is the boundary of some finite face of $G$.
For convenience, we do not distinguish a face cycle with its finite face. 

\begin{thm}\label{thm-4-1}
Let $G$ be a connected plane bipartite graph with perfect matchings.
If for each nice cycle $C$ of $G$ and a maximum resonant set $K$ of $I[C]$,
$I[C]-K$ has a unique perfect matching, then $Cl(G)=F(G)$.
\end{thm}
\begin{proof}
Let $F(G)=n$. By the definition of Clar number and Theorem \ref{thm-1-2}, we can see that $Cl(G)\leq n$.
In the following we show $Cl(G)\geq n$.
Define $\mathcal{M}(G)$ as the set of perfect matchings of $G$
whose forcing numbers equal $n$.
By Theorem \ref{thm-1-2}, for any $M\in \mathcal{M}(G)$,
there exist $n$ pairwise disjoint $M$-alternating cycles in $G$.
We choose a perfect matching $M_1$ in $\mathcal{M}(G)$ such that $n$ disjoint $M_1$-alternating cycles
$C_1,C_2,\ldots,C_n$ of $G$ have face cycles as many as possible.

Put $\mathcal{C}=\{C_1,C_2,\ldots,C_n\}$. It suffices to show that all cycles in $\mathcal{C}$ are face cycles.
Otherwise, $\mathcal{C}$ has a non-face cycle member and its interior
contains only face cycle members of $\mathcal{C}$.
Without loss of generality, let  $C_i$ denote such a non-face cycle member of $\mathcal{C}$
and $C_1,C_2,\ldots,C_{i-1}$ are all the face cycles in $\mathcal{C}$  contained in the interior of $C_i$
for some $i$, $1\leq i\leq n$.
Then the restriction of $M_1$ on $I[C_{i}]$
is also a perfect matching of $I[C_{i}]$, denoted by $M_c$.
By the assumption, $\{C_1,C_2,\ldots,C_{i-1}\}$
is  a non-maximum resonant set of $I[C_{i}]$.
Let $S$ be a maximum resonant set of $I[C_{i}]$. Then $|S|\geq i$.
Let $M_0$ be a perfect matching of $I[C_{i}]$
such that all faces in $S$ are $M_0$-resonant.
Let $M_2=(M_1\backslash M_c)\cup M_0$ and
$\mathcal{C}'= S\cup \{C_{i+1},C_{i+2},\ldots,C_n\}$.
Then $M_2$ is a perfect matching of $G$ and each member of $\mathcal{C}'$ is an $M_2$-alternating cycle.
Note that $M_2\in \mathcal{M}(G)$ and $\mathcal{C}'$ contain more face cycles than $\mathcal{C}$.
This contradicts the choices of $M_1$ and $\{C_1,C_2,\ldots,C_n\}$.
\end{proof}

Combining Theorems \ref{thm-4-1} with \ref{thm-2-1} and \ref{thm-1-1}, we immediately obtain the following results.

\begin{cor}\label{cor-4-2}
\emph{\cite{23}} Let $H$ be a hexagonal system with a perfect matching. Then $F(H)=Cl(H)$.
\end{cor}

\begin{cor}\label{cor-4-3}
Let $P$ be a polyomino graph with a perfect matching. Then $Cl(P)=F(P)$.
\end{cor}

We now give a weakly elementary property of such graphs that satisfies the conditions of Theorem \ref{thm-4-1}.
Let $G$ be a connected plane bipartite graph with a perfect matching. An edge of $G$ is called \emph{allowed} if it lies in
some perfect matching of $G$ and \emph{forbidden} otherwise. $G$ is called \emph{elementary} if each edge of $G$ is allowed.
$G$ is said to be \emph{weakly elementary} if for each nice cycle $C$ of $G$ the interior of $C$ has at
least one allowed edge of $G$ that is incident with a vertex of $C$ whenever the interior
of $C$ contains an edge of $G$ \cite{27}. A face $f$ of $G$ is said to be a \emph{boundary face}
if the boundaries of $f$ and $\partial (G)$ have a vertex in common.

\begin{thm}
Let $G$ be a connected plane bipartite graph with perfect matchings. If
for each nice cycle $C$ of $G$ and a maximum resonant set $K$ of $I[C]$, $I[C]-K$ has a unique
perfect matching, then $G$ is weakly elementary.
\end{thm}
\begin{proof}Suppose to the contrary that $G$ is not weakly elementary.
Then there exists a nice non-face cycle $C$ of $G$ such that the interior of $C$
has no allowed edges of $G$  incident with vertices of $C$.
It follows that the interior of $C$ has no allowed edges of $I[C]$
are incident with vertices of $C$. So, for every perfect matching $M$ of $I[C]$,
$C$ is $M$-alternating and any maximum resonant set $K$ of $I[C]$
contains no boundary faces of $I[C]$. So $I[C]-K$ has at least two perfect matchings,
a contradiction.
\end{proof}

Xu et al. ever gave  a conjecture as follows, which can be now confirmed.

\begin{conj}\label{conj-1}
\emph{\cite{23}} Let $G$ be an elementary polyomino graph.
Then the maximum forcing number of $G$ can be computed in polynomial time.
\end{conj}

Abeledo and  Atkinson  obatained the following result.

\begin{thm}\label{thm-4-6}
\emph{\cite{1}}  Let $G$ be a $2$-connected plane bipartite graph.
Then the Clar number of $G$ can be computed in polynomial time using linear programming methods.
\end{thm}

A polyomino graph $P$ with forbidden edges can be decomposed into some elementary components \cite{22}. A polynomial time algorithm to accomplish this decomposition. Such elementary components are elementary polyominoes and thus 2-connected. Note that the forcing number of any perfect matching $M$ of $P$ equals the sum of forcing numbers of the restrictions of $M$ on its elementary components. So
Theorem \ref{thm-4-6} and Corollary \ref{cor-4-3} imply the following result,  which  confirms Conjecture \ref{conj-1}.

\begin{thm}\label{thm-4-7}
Let $G$ be a  polyomino graph with a perfect matching.
Then the maximum forcing number of $G$ can be computed in polynomial time.
\end{thm}

\section{Concluding remarks}

Theorem \ref{thm-2-1} does not hold for general plane bipartite graphs even for plane elementary bipartite graphs.
Let us see  two elementary bipartite graphs $G$ and  $G'$  as shown in Fig. \ref{fig:weakly-elementary}, where  $G'$ is a subgraph of $G$ bounded by a nice cicle of $G$.
Since there exist at most $12$ (resp. $6$) pairwise disjoint finite faces in $G$ (resp. $G'$)
and the faces with labels $1$ (resp. $2$) form a resonant set of $G$ (resp. $G'$),
we have that $Cl(G)=12$ and  $Cl(G')=6$. There is a maximum resonant set $S'$ (the faces with labels $2$) of $G'$ such that
$G'-S'$ has two perfect matchings.  For any maximum resonant set $S$ of $G$,
$G-S$ is empty.  This example shows that Theorem \ref{thm-2-1} holds for a  graph $G$
does not imply that it holds for each such subgraph of $G$.
Since $G'$ has a perfect matching $M$ such that the faces with labels $2$ and the infinite face of $G'$ are $M$-resonant,
by Theorem \ref{thm-1-2} we have that $F(G')\geq 7$. Since there exist no $8$ pairwise disjoint cycles in $G'$,
we have that $F(G')\leq 7$.
So $F(G')=7$ and $F(G')\neq Cl(G')$. But $F(G)=Cl(G)=12$.

\begin{figure}[!htbp]
\begin{center}
\includegraphics[height=4.0cm]{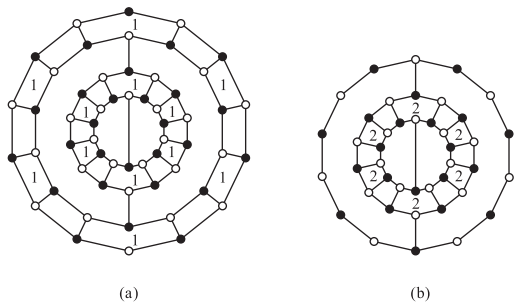}
\caption{(a) An elementary graph $G$,(b) A subgraph $G'$ of $G$.}
\label{fig:weakly-elementary}
\end{center}
\end{figure}


\end{document}